\makeatletter\let\ifGm@compatii\relax\makeatother

\documentclass[10pt]{amsart}
\usepackage{latexsym}
\usepackage{amssymb}
\usepackage{amscd}
\usepackage{epsf}
\usepackage{amsmath,amssymb,graphicx,amsthm,mathrsfs,verbatim}
\usepackage{pstricks}

\begin{document}

\theoremstyle{plain}
\newtheorem{Thm}{Theorem}[section]
\newtheorem{TitleThm}[Thm]{}
\newtheorem{Corollary}[Thm]{Corollary}
\newtheorem{Proposition}[Thm]{Proposition}
\newtheorem{Lemma}[Thm]{Lemma}
\newtheorem{Conjecture}[Thm]{Conjecture}
\theoremstyle{definition}
\newtheorem{Definition}[Thm]{Definition}
\theoremstyle{definition}
\newtheorem{Example}[Thm]{Example}
\newtheorem{TitleExample}[Thm]{}
\newtheorem{Remark}[Thm]{Remark}
\newtheorem{SimpRemark}{Remark}
\renewcommand{\theSimpRemark}{}

\numberwithin{equation}{section}

\newcommand{\C}{{\mathbb C}}
\newcommand{\Q}{{\mathbb Q}}
\newcommand{\R}{{\mathbb R}}
\newcommand{\Z}{{\mathbb Z}}
\newcommand{\mbS}{{\mathbb S}}
\newcommand{\mbU}{{\mathbb U}}
\newcommand{\mbO}{{\mathbb O}}
\newcommand{\mbG}{{\mathbb G}}
\newcommand{\mbH}{{\mathbb H}}

\newcommand{\flushpar}{\par \noindent}

\newcommand{\proj}{{\rm proj}}
\newcommand{\coker}{{\rm coker}\,}
\newcommand{\Sol}{{\rm Sol}}
\newcommand{\supp}{{\rm supp}\,}
\newcommand{\codim}{{\operatorname{codim}}}
\newcommand{\sing}{{\operatorname{sing}}}
\newcommand{\Tor}{{\operatorname{Tor}}}
\newcommand{\Hom}{{\operatorname{Hom}}}
\newcommand{\wt}{{\operatorname{wt}}}
\newcommand{\graph}{{\operatorname{graph}}}
\newcommand{\rk}{{\operatorname{rk}}}
\newcommand{\dlog}{{\operatorname{Derlog}}}
\newcommand{\Olog}[2]{\Omega^{#1}(\text{log}#2)}
\newcommand{\produnion}{\cup \negmedspace \negmedspace 
\negmedspace\negmedspace {\scriptstyle \times}}
\newcommand{\pd}[2]{\dfrac{\partial#1}{\partial#2}}

\def \ba {\mathbf {a}}
\def \bb {\mathbf {b}}
\def \bc {\mathbf {c}}
\def \bd {\mathbf {d}}
\def \bone {\boldsymbol {1}}
\def \bg {\mathbf {g}}
\def \bG {\mathbf {G}}
\def \bh {\mathbf {h}}
\def \bi {\mathbf {i}}
\def \bj {\mathbf {j}}
\def \bk {\mathbf {k}}
\def \bK {\mathbf {K}}
\def \bm {\mathbf {m}}
\def \bn {\mathbf {n}}
\def \bt {\mathbf {t}}
\def \bu {\mathbf {u}}
\def \bv {\mathbf {v}}
\def \by {\mathbf {y}}
\def \bV {\mathbf {V}}
\def \bx {\mathbf {x}}
\def \bw {\mathbf {w}}
\def \b1 {\mathbf {1}}
\def \bga {\boldsymbol \alpha}
\def \bgb {\boldsymbol \beta}
\def \bgg {\boldsymbol \gamma}

\def \itc {\text{\it c}}
\def \ite {\text{\it e}}
\def \ith {\text{\it h}}
\def \iti {\text{\it i}}
\def \itj {\text{\it j}}
\def \itm {\text{\it m}}
\def \itM {\text{\it M}} 
\def \itn {\text{\it n}}
\def \ithn {\text{\it hn}}
\def \itt {\text{\it t}}

\def \cA {\mathcal{A}}
\def \cB {\mathcal{B}}
\def \cC {\mathcal{C}}
\def \cD {\mathcal{D}}
\def \cE {\mathcal{E}}
\def \cF {\mathcal{F}}
\def \cG {\mathcal{G}}
\def \cH {\mathcal{H}}
\def \cK {\mathcal{K}}
\def \cL {\mathcal{L}}
\def \cM {\mathcal{M}}
\def \cN {\mathcal{N}}
\def \cO {\mathcal{O}}
\def \cP {\mathcal{P}}
\def \cS {\mathcal{S}}
\def \cT {\mathcal{T}}
\def \cU {\mathcal{U}}
\def \cV {\mathcal{V}}
\def \cW {\mathcal{W}}
\def \cX {\mathcal{X}}
\def \cY {\mathcal{Y}}
\def \cZ {\mathcal{Z}}

\def \ga {\alpha}
\def \gb {\beta}
\def \gg {\gamma}
\def \gd {\delta}
\def \ge {\epsilon}
\def \gevar {\varepsilon}
\def \gk {\kappa}
\def \gl {\lambda}
\def \gs {\sigma}
\def \gt {\tau}
\def \gw {\omega}
\def \gz {\zeta}
\def \gG {\Gamma}
\def \gD {\Delta}
\def \gL {\Lambda}
\def \gS {\Sigma}
\def \gW {\Omega}

\def \dim {{\rm dim}\,}
\def \mod {{\rm mod}\;}
\def \rank {{\rm rank}\,}

\newcommand{\ds}{\displaystyle}
\newcommand{\vf}{\vspace{\fill}}
\newcommand{\vect}[1]{{\bf{#1}}}
\def\R{\mathbb R}
\def\C{\mathbb C}
\def\CP{\mathbb{C}P}
\def\RP{\mathbb{R}P}
\def\N{\mathbb N}

\def\Sym{\mathrm{Sym}}
\def\Sk{\mathrm{Sk}}
\def\GL{\mathrm{GL}}
\def\Diff{\mathrm{Diff}}
\def\id{\mathrm{id}}
\def\Pf{\mathrm{Pf}}
\def\sll{\mathfrak{sl}}
\def\g{\mathfrak{g}}
\def\h{\mathfrak{h}}
\def\k{\mathfrak{k}}
\def\t{\mathfrak{t}}
\def\OcN{\mathscr{O}_{\C^N}}
\def\Ocn{\mathscr{O}_{\C^n}}
\def\Ocm{\mathscr{O}_{\C^m}}
\def\Ocnz{\mathscr{O}_{\C^n,0}}
\def\Derlog{\mathrm{Derlog}\,}
\def\expdeg{\mathrm{exp\,deg}\,}

\title[Characteristic Cohomology I]
{Characteristic Cohomology I: Singularities of Given Type}
\author[James Damon]{James Damon}

\address{Department of Mathematics, University of North Carolina, Chapel 
Hill, NC 27599-3250, USA
}



\begin{abstract}
For a germ of a variety $\cV, 0 \subset \C^N, 0$, a singularity $\cV_0$ of \lq\lq type $\cV$\rq\rq\, is given by a germ $f_0 : \C^n, 0 \to \C^N, 0$ which is transverse to $\cV$ in an appropriate sense so that $\cV_0 = f_0^{-1}(\cV)$.  If $\cV$ is a hypersurface germ, then so is $\cV_0 $, and by transversality $\codim_{\C} \sing(\cV_0) = \codim_{\C} \sing(\cV)$ provided $n > \codim_{\C} \sing(\cV)$.  So $\cV_0, 0$ will exhibit singularities of $\cV$ up to codimension $n$.  \par 
For singularities $\cV_0, 0$ of type $\cV$, we introduce a method to capture the contribution of the topology of $\cV$ to that of $\cV_0$.  It is via the \lq\lq characteristic cohomology\rq\rq\, of the Milnor fiber (for $\cV, 0$ a hypersurface), and complement and link of $\cV_0$ (in the general case).  The characteristic cohomology of the Milnor fiber $\cA_{\cV}(f_0; R)$, respectively of the complement $\cC_{\cV}(f_0; R)$, are subalgebras of the cohomology of the Milnor fibers, respectively the complement, with coefficients $R$ in the corresponding cohomology. For a fixed $\cV$, they are functorial over the category of singularities of type $\cV$.  In addition, for the link of $\cV_0$ there is a characteristic cohomology subgroup $\cB_{\cV}(f_0, \bk)$ of the cohomology of the link over a field $\bk$ of characteristic $0$.  The cohomologies $\cC_{\cV}(f_0; R)$ and $\cB_{\cV}(f_0, \bk)$ are shown to be invariant under the $\cK_{\cV}$-equivalence of defining germs $f_0$, and likewise $\cA_{\cV}(f_0; R)$ is shown to be invariant under the $\cK_{H}$-equivalence of $f_0$ for $H$ the defining equation of $\cV, 0$.  
\par
We give a geometric criterion involving \lq\lq vanishing compact models\rq\rq\, for both the Milnor fibers and complements which detect nonvanishing subalgebras of the characteristic cohomologies, and subgroups of the characteristic cohomology of the link.
Also, we consider how in the hypersurface case the cohomology of the Milnor fiber as a module over the characteristic cohomology $\cA_{\cV}(f_0; R)$. We briefly consider the application of these results to a number of cases of singularities of a given type.  In part II we specialize to the case of matrix singularities and using results on the topology of the Milnor fibers, complements and links of the varieties of singular matrices obtained in another paper allow us to give precise results for the characteristic cohomology of all three types.  
\end{abstract} 
\maketitle
\vspace{2ex}
\centerline{\bf Preliminary Version}
\vspace{1ex}
\par
\section*{Introduction}  
\label{S:sec0} 
\par
For a germ of a hypersurface $\cV_0, 0 \subset \C^n, 0$ with a nonisolated singularity, a result of Kato-Masumoto \cite{KMs} states that the connectivity of the Milnor fiber may decrease by $r = \dim_{\C} \sing(\cV_0)$.  Thus, it may have nonzero (co)homology $H^j(\cV_0)$ in dimension $n -1 - r \leq j \leq n-1$.  For very low dimensional singular sets of dimension $\leq 2$, with special forms for $\sing(\cV_0)$ and the transverse types of the defining equation $f_0$ on $\sing(\cV_0)$, the work of Siersma and coworkers Pellikan, Tibar, Nemethi, Zaharia, Van Straten, etc., have determined the topological structure of the Milnor fibers (see e.g. the survey \cite{Si}). However, very little is known about the topology for hypersurfaces with higher dimensional singular sets.  We consider in this paper how we may introduce in such a situation more information about the topology of a singularity $\cV_0$, which is based on a \lq\lq universal singularity\rq\rq\, $\cV$, even when it is highly nonisolated.  This will be done by identifying how topological properties of $\cV$ are inherited by $\cV_0$.   \par
	We give a general formulation for the category of singularities $\cV_0$ of 
\lq\lq type $\cV$\rq\rq\, for a fixed germ of a variety $\cV, 0 \subset \C^N, 0$ defined as 
$\cV_0 = f_0^{-1}(\cV)$ for a germ $f_0 : \C^n, 0 \to \C^N, 0$ (which for a subcategory is transverse to $\cV$ in an appropriate sense).  If $\cV$ is a hypersurface germ, then so is $\cV_0$.  If $\cV, 0$ is a highly singular germ and $n > \codim_{\C} \sing(\cV)$, then by transversality, $\cV_0, 0$ will also exhibit singularities of $\cV$ up to codimension $n$, and hence also in general be highly singular.  Nonetheless we define the characteristic cohomology for the Milnor fiber (for the hypersurface case), and the complement and link of $\cV_0$ (in the general case). 
\par
The \lq\lq characteristic cohomology algebra\rq\rq\ of the Milnor fiber of $\cV_0$ is defined as $\cA_{\cV}(f_0; R) = \widetilde{f_0}^*(H^*(F_w; R)$, for 
$\widetilde{f_0}; \cV_w \to F_w$ the induced map of Milnor fibers.  Likewise, the \lq\lq characteristic cohomology algebra\rq\rq\ of the link is defined to be $\cC_{\cV}(f_0; R) = f_0^*(H^*(\C^N \backslash \cV; R)$ (which is understood in the sense of local cohomology).  Both of these are shown to be well-defined and functorial over the category of singularities of type $\cV$ for a fixed singularity $\cV$.  For a field $\bk$ of characteristic $0$, the \lq\lq characteristic cohomology (subspace)\rq\rq\ of the link, 
$\cB_{\cV}(f_0; \bk)$ is defined to be the Alexander dual of the Kronecker dual of 
$\cC_{\cV}(f_0; \bk)$.  It is not functorial, but is natural with respect to a relative form of the Gysin homomorphism.  \par
We show that $\cA_{\cV}(f_0; R)$ is invariant, up to an algebra isomorphism of the cohomology of the Milnor fiber, under $\cK_H$-equivalence of $f_0$  (i.e. 
$\cK$-equivalence of $f_0$ preserving the defining equation $H$ of $\cV$, see e.g. \cite{DM}).  Also, both $\cC_{\cV}(f_0; R)$ and $\cB_{\cV}(f_0; \bk)$ are invariant 
under $\cK_{\cV}$-equivalence of $f_0$, up to an algebra isomorphism of the cohomology of the complement, resp. the isomorphism of the cohomology group of the link.  This will allow us to give a structural form for the cohomology of the Milnor fiber (in the hypersurface case) and of the complement (for general $\cV$), as modules over corresponding \lq\lq characteristic subalgebras\rq\rq.  Furthermore, we give results about the exact form of these characteristic subalgebras.  \par 
In Part II \cite{D6}, we will give results for categories of matrix singularities where $\cV$ denotes any of the varieties of singular $m \times m$ complex matrices which may be either general, symmetric or skew-symmetric (with $m$ even) and for $m \times p$ matrices with $m \neq p$.  These give rise to \lq\lq matrix singularities\rq\rq\, $\cV_0$ of any of the corresponding types.  For matrix singularities the characteristic cohomology will give the analogue of characteristic classes for vector bundles.
\par
In \S \ref{S:sec7} we begin to investigate for hypersurface singularities how the cohomology of the Milnor fiber can be understood as a module over the characteristic cohomology subalgebra and the role that the topology of the singular Milnor fiber plays.  This is further considered  for examples in \S \ref{S:sec11}. 
\par
Lastly, we consider in \S \ref{S:sec11} a number of general classes of nonisolated complex singularities which are of a given \lq\lq universal type\rq\rq\,.  These include discriminants of finitely determined (holomorphic) map germs; bifurcation sets for 
$\cG$-equivalence where $\cG$ is a geometric subgroup of $\cA$ or $\cK$ in the holomorphic category; generic hyperplane or hypersurface arrangements based on special central complex hyperplane arrangements, and determinantal arrangements arising from exceptional orbit varieties of prehomogeneous spaces (which includes matrix singularities).  We consider how specific results for these examples reveal the role that characteristic cohomology is playing for these cases.  
\par 
\vspace{1ex}
\centerline{CONTENTS}
\vspace{1ex}
\begin{enumerate}
\item  Characteristic Cohomology of Singularities of type $\cV$ 
\par 
\vspace{2ex} 
\par
\item  $\cK_H$ and $\cK_{\cV}$ Invariance of Characteristic Cohomology 
\par 
\vspace{2ex} 
\par\item Detecting the Nonvanishing of Characteristic Cohomology
\par 
\vspace{2ex} 
\par
\item	Module Structure for the Cohomology of Milnor Fibers of Matrix Singularities
\par 
\vspace{2ex}
\par
\item	Detecting Characteristic Cohomology for Various General Cases

\end{enumerate}
\vspace{3ex}

\section{Characteristic Cohomology of Singularities of type $\cV$} 
\label{S:sec1}
We begin by considering singularities arising as nonlinear sections of some given \lq\lq universal \rq\rq\, singularity $\cV, 0$.  There are many fundamental examples of such universal singularities which are, in particular, hypersurface singularities including: reflection hyperplane arrangements, discriminants of stable map germs, bifurcation sets for the $\cG$-versal unfoldings of germs for many different singularity equivalence groups $\cG$ which are \lq\lq geometric subgroups of $\cA$ or $\cK$\rq\rq\, (see e.g. \cite{D2} and papers cited therein), exceptional orbit hypersurfaces of prehomogeneous spaces \cite{D4} which include both reductive groups, e.g. \cite{BM}, and solvable groups \cite{DP2}, \cite{DP3}, as well as specifically the varieties of singular $m \times m$ matrices which may be general, symmetric, or skew-symmetric (if $m$ is even).  There are also other classes of universal singularities which are not hypersurface singularities, such as bifurcation sets for certain $\cG$-versal unfoldings and varieties of singular $m \times p$ matrices with $m \neq p$.  The results for complements and links will also be applicable to the non-hypersurface cases.
\par 
\subsection*{Category of Singularities of Type $\cV$} \hfill 
\par
We recall from \cite{D5}  that given a germ of an analytic set $\cV , 0 \subset \C^N, 0$, a \lq\lq nonlinear section\rq\rq  is given by a germ of a holomorphic map $f_0 : \C^n, 0 \to \C^N, 0$ (so that $f_0(\C^n) \not \subset \cV$), where $n$ may take any value (including allowing $n > N$).  The associated singularity of type $\cV$ is $\cV_0 = f_0^{-1}(\cV)$.  
\begin{equation}
\label{tag1.1}
\begin{CD}
  @.  {\C^n,0} @>{f_0}>> \C^N,0\\
@.  @AAA @AAA \\
{f_0^{-1}(\cV)}  @=  \cV_0, 0  @>>> \cV, 0  
\end{CD} 
\end{equation}
\par
We consider the {\em category of singularities of type $\cV$}. The objects are the singularities of type $\cV$.  
Given two singularities of type $\cV$: $\cV_0$ defined by $f_0 : \C^n, 0 \to \C^N, 0$ and $\cW_0$ defined by $g_0 : \C^s, 0 \to \C^N, 0$, a morphism $\psi : \cW_0, 0 \to \cV_0, 0$ is given by a germ $\tilde{\psi} : \C^s, 0 \to \C^n, 0$  such that $g_0 = f_0 \circ \tilde{\psi}$.  Such singularities of type $\cV$ and the corresponding morphisms between them give a category on which we will define the characteristic cohomology.  \par
The basic equivalence for studying the ambient equivalence of such $\cV_0$ is 
$\cK_{\cV}$-equivalence of the germs $f_0$, which is a form of $\cK$-equivalence which preserves $\cV$, see e.g. \cite{D5} or \cite{D2}.  This equivalence applied to an $f_0 : \C^n, 0 \to \C^N, 0$ can be viewed as the action on the section $graph(f_0) : \C^n, 0 \to \C^n \times\C^N, 0$ of the trivial vector bundle on $\C^n$ with fiber $\C^N$.  It acts via diffeomorphisms of the fibers preserving each copy of $\cV$ and which holomorphically varies pointwise on $\C^n, 0$ composed with a local diffeomorphism of $\C^n, 0$.  As such it is a type of gauge group.  \par 
We also consider the defining equation $H : \C^N, 0 \to \C, 0$ for $\cV$.  There is a stronger $\cK_H$-equivalence within $\cK_V$-equivalence (see \cite{DM} and \cite{D1}) where the diffeomorphisms of $\C^n \times \C^N, (0, 0)$ preserve the defining map germ $H \circ pr_2 : \C^n \times \C^N, (0, 0) \to \C, 0$ for $\C^n \times \cV, (0, 0)$, where $pr_2$ denotes projection onto the second factor $\C^N, 0$.  These diffeomorphisms not only preserve $\C^n \times \cV$, but also $\C^n \times F$ for $F$ a Milnor fiber of $\cV$. \par
We further consider a subcategory of singularities of type $\cV$ where the germ $f_0$ is transverse to $\cV$ on the complement of $0$ in $\C^n$.  Transversality can be either in a geometric sense of transversality to the canonical Whitney stratification of $\cV$ or in an algebraic sense using the module of logarithmic vector fields (see \cite{D1}) and these agree if $\cV$ is holonomic in the sense of Saito \cite{Sa}.  In these cases the corresponding germ is finitely $\cK_{\cV}$-determined.  These singularities and the corresponding morphisms between them give a subcategory of \lq\lq finitely determined singularities of type $\cV$\rq\rq.  
 \par 
In analyzing the topology of such singularities $\cV_0$ there are three contributions:
\flushpar
\begin{itemize}
\item[a)]  the contribution from the topology of the germ $f_0$ and its 
geometric interaction with $\cV$;
\item[b)] the contribution from the topology of $\cV$;
\item[c)]  the interaction between these two contributions combining to give the 
topology of $\cV_0$.
\end{itemize}
\par  
For a), there have been results introduced for discriminants of finitely determined mappings and more generally finitely determined nonlinear sections of free divisors and complete intersections in \cite{DM} and \cite{D1}, and of the varieties of $m \times m$ matrices in \cite{GM} and \cite{DP3}, using a stabilization of the mapping to obtain a \lq\lq singular Milnor fiber\rq\rq\, homotopy equivalent to a bouquet of spheres, with the number of such spheres computed algebraically.  However, this provides no information about b).  The characteristic cohomology which we will introduce will specifically address b) and provide complementary information to that given for a).  We briefly indicate in \S \ref{S:sec7} how these two contributions combine for c).  \par
\subsection*{Characteristic Cohomology on the Category of Singularities of Type $\cV$} \hfill 
\par
We begin with the definition for the Milnor fiber in the case $\cV, 0$ is a hypersurface singularity.  
\subsubsection*{Characteristic Cohomology $\cA_{\cV}(f_0, R)$} \hfill 
\par
Let $f_0 : \C^n, 0 \to \C^N, 0$ define the singularity $\cV_0$.  For $\cV$ there exists $0 < \eta << \gd$ such that for balls $B_{\eta} \subset \C$ and $B_{\gd} \subset \C^N$  (with all balls centered $0$), we let
 $\cF_{\gd} = H^{-1}(B_{\eta}^*) \cap B_{\gd}$ 
so $H : \cF_{\gd} \to B_{\eta}^*$ is the Milnor fibration of $H$, with Milnor fiber 
$F_w = H^{-1}(w) \cap B_{\gd}$ for each $w \in B_{\eta}^*$.  By continuity, 
there is an $\gevar > 0$ so that $f_0(B_{\gevar}) \subset 
\cF_{\gd}$.  By possibly shrinking all three values, 
$H \circ f_0 : f_0^{-1}(\cF_{\gd}) \cap B_{\gevar} \to B_{\eta}^*$ is the 
Milnor fibration of $H \circ f_0$.  
Then, the Milnor fiber of $H \circ f_0$ for $w \in B_{\eta}^*$ is given by 
$$\cV_w \,\, = \,\, (H \circ f_0)^{-1}(w) \cap B_{\gevar} \,\, = \,\, f_0^{-1}(F_w) \cap
 B_{\gevar}\, .$$  
\par
Thus, if we denote $f_0 | \cV_w = f_{0, w}$, then in cohomology with 
coefficient ring $R$, $f_{0, w}^* : H^*(F_w ; R) \to H^*(\cV_w ; R)$.  
we let
\begin{equation}
\label{Eqn1.3}
\cA_{\cV}(f_0; R) \,\, \overset{def}{=} \,\, f_{0, w}^* (H^*(F_w ; R))\, ,
\end{equation}
We formally define the characteristic cohomology of the Milnor fiber.
\begin{Definition}
\label{Def1.1}
Let $f_0 : \C^n, 0 \to \C^N, 0$ define $\cV_0 = f_0^{-1}(\cV)$.  We define the {\em characteristic cohomology subalgebra of the Milnor fiber} of $\cV_0$, to be cohomology subalgebra of the Milnor fiber $H^*(\cV_w ; R)$ of $\cV_0$ given by \eqref{Eqn1.3}.
\end{Definition}  
\subsubsection*{Independence of $\cA_{\cV}(f_0, R)$ on the Milnor Fiber under Cohomology Isomorphism} \hfill 
\par
Given another $w^{\prime} \in B_{\eta}^*$, let $\gg(t)$ denote a simple path in $B_{\eta}^*$ from $w$ to $w^{\prime}$.  We may first lift $\gg(t)$ to an isotopy $\Phi_t : F_w \to F_{\gg(t)}$ of the restriction of the Milnor fibration from $F_w$ to 
$F_{w^{\prime}}$.  We can also lift $\gg(t)$ to an isotopy 
$\Psi_t : \cV_w \to \cV_{\gg(t)}$ of the restriction of the Milnor fibration from $\cV_w$ to 
$\cV_{w^{\prime}}$.  Then, $\Phi_t^{-1} \circ f_0 \circ \Psi_t : \cV_w \to F_w$ defines a homotopy from $f_{0, w}$ to $\Phi_1^{-1} \circ f_{0, w^{\prime}} \circ \Psi_1$.  Thus, $f_{0, w}^* = \Psi_1^* \circ f_{0, w^{\prime}}^* \circ \Phi_1^{*\, -1}$.  Then, 
$\Phi_1^{*\, -1} : H^*(F_{w^{\prime}}; R) \simeq H^*(F_w; R)$, and 
$\Psi_1^* : H^*(\cV_{w^{\prime}}; R) \simeq H^*(\cV_w; R)$.  Hence, $f_{0, w^{\prime}}^*(H^*(\cV_{w^{\prime}}; R))$ is mapped under the cohomology algebra isomorphism $\Psi_1^*$ to $f_{0, w}^*(H^*(\cV_w; R))$.  Thus, $\Psi_1^*$ maps the characteristic cohomology for the Milnor fiber of $\cV_w$ to that of $\cV_{w^{\prime}}$.  
\par
We also remark that if we consider a second set of values $0 < \eta^{\prime} < \eta$, $0 < \gd^{\prime} < \gd$, and $0 < \gevar^{\prime} < \gevar$ for the Milnor fibers of $H$ and $H\circ f_0$, and choose $w \in B_{\eta^{\prime}}^*$ so that the Minor fiber 
$\cV_w^{\prime}$ is transverse to the spheres $S^{2n-1}_{\gevar^{\prime\prime}}$ for 
$\gevar^{\prime} < \gevar^{\prime\prime} < \gevar$, then $\iti_w : \cV_w^{\prime} \subset \cV_w$  is a homotopy equivalence so the characteristic cohomology for 
$\cV_w^{\prime}$ is mapped isomorphically to that of $\cV_w$.  Hence, the characteristic cohomology is well-defined independent of the Milnor fiber up to Milnor fiber cohomology isomorphism.  When we want to refer to the characteristic cohomology at more than one point $w \in B^*_{\eta}$, we use the notation $\cA(f_0, R)_{w}$ to denote the representative in the Milnor fiber cohomology $H^*(\cV_w; R)$. 
\begin{Remark}
\label{Rem1.1}
We consider two consequences of the above arguments.  First, if we choose a convex neighborhood $w \in U \subset B_{\eta}^*$, then as the paths in $U$ between $w$ and any other $w^{\prime}$ are homotopic, it follows that the induced diffeomorphisms between the Milnor fibers $\cV_w$ and $\cV_{w^{\prime}}$, resp. $F_w$ and $F_{w^{\prime}}$, are homotopic so the algebra isomorphisms between the cohomology of the Milnor fibers over $U$ is well-defined.  This gives a local trivialization of the unions $\cup_{w^{\prime} \in U} \cA(f_0, R)_{w^{\prime}}$, resp. $\cup_{w^{\prime} \in U} H^*(\cV_{w^{\prime}}; R)$.  On overlaps of two such neighborhoods the transition isomorphisms are constant.  Together they give a locally constant system on 
$B_{\eta}^*$.  Second, if $\gg(t)$ is a simple loop in $B_{\eta}^*$ from $w$ around $0$, then the preceding arguments show the monodromy will map the characteristic cohomology to itself.  Thus, the characteristic cohomology inherits two properties from the Milnor fiber cohomology.  In this paper we will not attempt to make use of these additional properties.  
\end{Remark}
\subsubsection*{Characteristic Cohomology $\cC_{\cV}(f_0, R)$} \hfill 
\par
We next introduce the characteristic cohomology of the complement of $\cV_0$ in the case where $\cV, 0$ need not be a hypersurface singularity.  This proceeds somewhat analogously to the case of Milnor fibers.  Let $f_0 : \C^n, 0 \to \C^N, 0$ 
define $\cV_0 = f_0^{-1}(\cV)$.  Then, we consider a representative $\tilde{f}_0 : U \to W$ for which $\cV$ also has a representative on $W$, and we still denote the representative by $\cV$.  Then, $\tilde{f}_0^{-1}(\cV)$ is a representative for $\cV_0$ which we still denote by $\cV_0$.  Then, by stratification theory (see e.g. Mather \cite{M1}, \cite{M2} or Gibson et al \cite{GDW}), there are $0 < \gd_0, \gevar_0$ so that for $0 < \gd^{\prime} < \gd \leq \gd_0$ and $0 < \gevar^{\prime} < \gevar \leq \gevar_0$:
\begin{itemize}
\item[i)]  $\overline{B}_{\gd_0} \subset W$ and $\overline{B}_{\gevar_0} \subset U$,
\item[ii)]  $\partial\overline{B}_{\gd}$ is transverse to $\cV$ and 
$\partial\overline{B}_{\gevar}$ is transverse to $\cV_0$.
\item[iii)]  $\cV_0 \cap \overline{B}_{\gevar^{\prime}}$ is ambiently homeomorphic to the cone on $\cV_0 \cap \partial\overline{B}_{\gevar^{\prime}}$, as is 
$\cV \cap \overline{B}_{\gd^{\prime}}$ ambiently homeomorphic to the cone on 
$\cV \cap \partial\overline{B}_{\gd^{\prime}}$, and 
\item[iv)]  the inclusions of pairs 
$$(\overline{B}_{\gevar^{\prime}}, \cV_0 \cap \overline{B}_{\gevar^{\prime}}) \hookrightarrow (\overline{B}_{\gevar}, \cV_0 \cap \overline{B}_{\gevar}) $$
and
$$ (\overline{B}_{\gd^{\prime}}, \cV \cap \overline{B}_{\gd^{\prime}}) \hookrightarrow (\overline{B}_{\gd}, \cV \cap \overline{B}_{\gd})\, , $$
are homotopy equivalences.
\end{itemize}
Thus, if  
$f_0(\overline{B}_{\gevar^{\prime}}) \subset B_{\gd^{\prime}}$,  and $f_0(\overline{B}_{\gevar}) \subset B_{\gd}$, then there is the commutative diagram  
\begin{equation}
\label{CD1.5}
\begin{CD}
   {H^*(\overline{B}_{\gd}\backslash \cV ;R)} @> {f_0^*} >> {H^*(\overline{B}_{\gevar} \backslash \cV_0 ; R)}\\
   @V{\simeq}VV @V {\simeq}VV \\
H^*(\overline{B}_{\gd^{\prime}}\backslash \cV ;R) @>{f_0^*}>> H^*(\overline{B}_{\gevar^{\prime}} \backslash \cV_0 ; R)  \\
\end{CD} 
\end{equation}
and the vertical maps are isomorphisms by property (iv).  Thus, via the vertical isomorphisms, the induced homomorphisms 
$f_0^* : H^*(\overline{B}_{\gd}\backslash \cV ;R) \to H^*(\overline{B}_{\gevar} \backslash \cV_0 ; R)$ 
are independent of 
$0 < \gevar < \gevar_0$ and $0 < \gd < \gd_0$.  
Hence, the induced isomorphisms 
$f_0^*(H^*(\overline{B}_{\gd}\backslash \cV ;R)) \simeq f_0^*(H^*(\overline{B}_{\gd^{\prime}}\backslash \cV ;R))$ 
yield an inverse system with limit isomorphic to each of these groups, giving a 
well-defined cohomology subalgebra.  
\begin{Definition}
\label{Def1.2}
Let $f_0 : \C^n, 0 \to \C^N, 0$ define $\cV_0 = f_0^{-1}(\cV)$.  We define the {\em characteristic cohomology (algebra) of the complement} of $\cV_0$, to be cohomology subalgebra which is the direct limit 
$$  \cC_{\cV}(f_0, R) \,\, \overset{def}{=} \,\, \underset{\rightarrow}{\lim} f_0^*(H^*(\overline{B}_{\gd}\backslash \cV ; R))\, .$$
\end{Definition}
We note that this cohomology is really in local cohomology of the complement, but it is given by the complement in sufficient small neighborhoods. 
\par
Just as for complements, singularities $\cV_0$ of type $\cV$ also have characteristic cohomology for the link. \par 
\subsubsection*{Characteristic Cohomology $\cB_{\cV}(f_0, R)$} \hfill 
\par
We use the same notation as above for the complement where again $\cV, 0$ need not be a hypersurface singularity.  In this case, we consider $R = \bk$, a field of characteristic $0$. 
By the conical structure for the pair $(\overline{B}_{\gevar}, \overline{B}_{\gevar} \cap \cV_0)$, it follows that the inclusion $j_{\gevar} : (S^{2n-1}_{\gevar} \backslash \cV_0) \subset \overline{B}_{\gevar} \backslash \cV_0$ is a homotopy equivalence.  Thus, $j_{\gevar}^*: H^*(\overline{B}_{\gevar} \backslash \cV_0 ; \bk) \simeq H^*(S^{2n-1}_{\gevar} \backslash \cV_0 ; \bk)$ is an isomorphism.  
\par 
For each $0 < \gevar \leq \gevar_0$, there is the Kronecker dual graded subgroup of 
$$  j_{\gevar}^* \circ f_0^*(H^*(\overline{B}_{\gd}\backslash \cV ;\bk)) \,\, \subset \,\,  H^*(S^{2n-1}_{\gevar} \backslash \cV_0 ; \bk)\, , $$
 which we denote by $\gG_{\cV}(f_0; \bk) \subset H_*(S^{2n-1}_{\gevar} \backslash \cV_0 ; \bk)$.  We note that for the Kronecker pairing we may choose a dual basis for $H_*(\overline{B}_{\gevar} \backslash \cV_0 ; \bk)$  that extends a basis for 
$j_{\gevar}^* \circ f_0^*(H^*(\overline{B}_{\gd}\backslash \cV ;R))$, so it is dually paired to $\gG_{\cV}(f_0; \bk)$.  

Then, we can apply a form of Alexander duality for subspaces of spheres, \cite[Chap. XIV, Thm 6.6]{Ma} or see e.g. \cite[Prop. 1.9]{D3}.  For $L(\cV_0) = S^{2n-1}_{\gevar} \cap \cV_0$, the link of $\cV_0$,
\begin{equation}
\label{Eqn1.6} 
\ga :  \widetilde{H}^j(L(\cV_0) ; \bk) \simeq  \widetilde{H}_{2n-2-j}(S^{2n-1}_{\gevar} \backslash L(\cV_0); \bk)  \qquad \text{for all j} 
\end{equation}
Then, if $\widetilde{\gG}_{\cV}(f_0; \bk)$ denotes the corresponding reduced homology obtained by removing $H_0$ from $\gG_{\cV}(f_0; \bk)$, then we define the characteristic cohomology for the link.
\begin{Definition}
\label{Def1.3}
Let $f_0 : \C^n, 0 \to M, 0$ define $\cV_0 = f_0^{-1}(\cV)$.  We define the {\em characteristic cohomology of the link} of $\cV_0$, to be
\begin{equation}
\label{Eqn1.7}
  \cB_{\cV}(f_0; \bk) \,\, \overset{def}{=} \,\, \ga^{-1}(\widetilde{\gG}_{\cV}(f_0); \bk)
\end{equation}
\end{Definition}
Since the definition in \eqref{Eqn1.7} is independent, up to isomorphism, of $\gevar$, this gives a well-defined graded cohomology subgroup in the cohomology of the link.  However, because of the use of Alexander duality, this is not a subalgebra as is the case for the Milnor fiber and the complement.  Also, the actual subgroup does depend upon the choice of basis for the Kronecker pairing; however, we still obtain subspaces in each degree whose dimensions are independent of choices.  
\begin{Remark}
\label{Rem1.8}
On first glance it might seem that it would be more natural to define the characteristic cohomology of the link to be 
$$ \cL_{\cV}(f_0, R) \,\, \overset{def}{=} \,\, \underset{\rightarrow}{\lim} f_0^*(H^*(\overline{B}_{\gd}\cap \cV ; R))\, .$$
However, we shall see in Part II \cite{D6} that this subgroup of the cohomology of the link does not capture the directly identifiable cohomology in $H^*(L(\cV_0), R)$.  Specifically this cohomology will lie above the middle dimension, while theorems such as the 
Le-Hamm Local Lefschetz Theorem, see e.g. \cite{HL} or 
\cite[Part 2, \S 1.2, Thm 1]{GMc}, when they are applicable only concern dimensions below the middle dimension.   
\end{Remark}
\par
\subsection*{Functoriality of Characteristic Cohomology $\cA_{\cV}(f_0, R)$ and 
$\cC_{\cV}(f_0, R)$} \hfill 
\par
We complete this section by establishing the functoriality of both $\cA_{\cV}(f_0, R)$ and $\cC_{\cV}(f_0, R)$ on the category of singularities of type $\cV$.
\begin{Lemma}
\label{Lem1.10}
Given $f_0 : \C^n, 0 \to \C^N, 0$ defining $\cV, 0$ and $g_0 : \C^s, 0 \to \C^N, 0$ defining $\cW, 0$ both of type $\cV, 0 \subset \C^N, 0$ with a morphism $\varphi : \cW_0, 0 \to \cV_0, 0$ defined by $\tilde{\varphi} : \C^s, 0 \to \C^n, 0$.  Then, $\varphi$ induces the algebra homomorphisms  
$\tilde{\varphi}^* : \cA_{\cV}(f_0, R) \to \cA_{\cV}(g_0, R)$ and 
$ \tilde{\varphi}^* : \cC_{\cV}(f_0, R) \to \cC_{\cV}(g_0, R)$.  Moreover, both $\cA_{\cV}(f_0, R)$ and $\cC_{\cV}(f_0, R)$ are functorial.
\end{Lemma}
\par 
Then, we shall let $\varphi^* : \cA_{\cV}(f_0, R) \to \cA_{\cV}(g_0, R)$ and 
$\varphi^* : \cC_{\cV}(f_0, R) \to \cC_{\cV}(g_0, R)$ denote the induced algebra homomorphisms defined by $\tilde{\varphi}^*$.  
\begin{proof}
\par
We begin by showing that $\tilde{\varphi}^*$ gives a well-defined homomorphism between the algebras in each case.  We consider $0 < \eta << \gevar_2, \gevar_1, \gd$ so that:
\begin{itemize}
\item[i)]  $\varphi(B_{\gevar_2}) \subset B_{\gevar_1}$, $f_0(B_{\gevar_1}) \subset B_{\gd}$, and $H(B_{\gd}) \subset B_{\eta}$; and
\item[ii)]  $H : H^{-1}(B_{\eta}^* ) \cap B_{\gd} \to B_{\eta}^*$ is the Milnor fibration for $H$; $H\circ f_0 : (H\circ f_0)^{-1}(B_{\eta}^*) \cap B_{\gevar_1} \cap  \to B_{\eta}^*$ is the Milnor fibration for 
$H\circ f_0$; and $H\circ f_0 \circ \tilde{\varphi} : (H \circ f_0 \circ \tilde{\varphi})^{-1}(B_{\eta}^*) \cap B_{\gevar_2} \cap  \to B_{\eta}^*$ is the Milnor fibration for $H\circ f_0\circ \tilde{\varphi} = H\circ g_0$.  
\end{itemize}
Then, for $w \in B_{\eta}^*$ we have the induced maps for the cohomology of the Milnor fibers
\begin{equation}
\label{Eqn1.11}
H^*(F_w; R) \overset{f_{0\, w}^*}{\longrightarrow} H^*(\cV_w; R) \overset{\tilde{\varphi}_{w}^*}{\longrightarrow} H^*(\cS_w; R) \, . 
\end{equation}
Then the composition in \eqref{Eqn1.11} is 
\begin{equation}
\label{Eqn1.12}
H^*(F_w; R) \overset{\tilde{\varphi}_{w}^* \circ f_{0\, w}^*}{\longrightarrow}  H^*(\cS_w; R) \, . 
\end{equation}
The image of this composition in \eqref{Eqn1.12} defines $\cA_{\cV}(g_0, R)$ and factors through \eqref{Eqn1.13}. 
Hence, $\tilde{\varphi}_{w}^*$ induces a well-defined map $\tilde{\varphi}^* : \cA_{\cV}(f_0, R) \to \cA_{\cV}(g_0, R)$.  
\begin{equation}
\label{Eqn1.13}
H^*(F_w; R) \overset{f_{0\, w}^*}{\longrightarrow}  H^*(\cV_w; R) \, . 
\end{equation}
\par 
For functoriality, we include a third singularity $\cZ_0$ of type $\cV$ given by $h_0 : \C^r, 0 \to \C^N, 0$ such that there is a map germ $\tilde{\psi} : \C^r, 0 \to \C^s, 0$ so that $g_0 \circ \tilde{\psi} = h_0$.  Then, choosing an additional $0 < \eta << \gevar_3$ so that $\tilde{\psi}(B_{\gevar_3}) \subset B_{\gevar_2}$, and $H\circ f_0 \circ \tilde{\varphi} \circ \tilde{\psi} : (H \circ f_0 \circ \tilde{\varphi} \circ \tilde{\psi})^{-1}(B_{\eta}^*) \cap B_{\gevar_3} \cap  \to B_{\eta}^*$ is the Milnor fibration for $H\circ f_0\circ \tilde{\varphi} \circ \tilde{\psi} = H\circ h_0$.
Then, by functoriality in cohomology, $\tilde{\psi}$ maps the image in \eqref{Eqn1.12} to 
$H^*(\cZ_w; R)$, for $\cZ_w$ the Milnor fiber of $H \circ h_0$ over $w$, and 
$(\tilde{\varphi} \circ \tilde{\psi})^* = \tilde{\psi}^*\circ \tilde{\varphi}^*$.  Hence, using our notation for the induced maps on characteristic cohomology,  $(\varphi \circ \psi)^* = \psi^*\circ \varphi^*$.  \par
For $\cC_{\cV}(f_0, R)$, $\cC_{\cV}(g_0, R)$, and $\cC_{\cV}(h_0, R)$, the proof is similar, except we replace the Milnor fibers by the complements $B_{\gd} \backslash \cV$, resp. $B_{\gevar_1} \backslash \cV_0$, resp. $B_{\gevar_2} \backslash \cW_0$, resp. $B_{\gevar_3} \backslash \cZ_0$ and consider the induced maps in cohomology of these complements by $f_0^*$, resp. $\tilde{\varphi}^*$, resp.  $\tilde{\psi}^*$ and their compositions. 
\end{proof}
One immediate consequence of functoriality is the detection of the nonvanishing characteristic cohomology.  We note for the identity map $id : \C^N, 0 \to \C^N, 0$, 
$\cA_{\cV}(id, R)_w = H^*(F_w; R)$.  With the above notation for a morphism $\varphi : \cW_0, 0 \to \cV_0, 0$ defined by $\widetilde{\varphi} : \C^p, 0 \to \C^n, 0$ with $\cV_0, 0$ defined by $f_0 : \C^n, 0 \to \C^N, 0$ and $\cW_0, 0$ defined by $g_0 : \C^p, 0 \to \C^N, 0$.  Then, we have the corollary.
\begin{Corollary}
\label{Cor1.14}
If $g_0^* : \cA_{\cV}(id, R) \to \cA_{\cV}(g_0, R)$ induces an isomorphism from a graded subgroup $E \subset \cA_{\cV}(id, R)$ to a subgroup of $\cA_{\cV}(g_0, R)$, then $f_0^*$ induces an isomorphism from $E$ to a subgroup of $\cA_{\cV}(f_0, R)$.  \par
There is an analogous result for $\cC_{\cV}(f_0, R)$ and the complement. 
\end{Corollary}
\begin{proof}
By functoriality, we have for the sequence 
$$ \cA_{\cV}(id, R)  \overset{f_0^*}{\longrightarrow} \cA_{\cV}(g_0, R) \overset{\varphi^*}{\longrightarrow} \cA_{\cV}(g_0, R) $$
the composition is $\varphi^*\circ f_0^* = g_0^*$.  As $g_0^*$ maps $E$ isomorphically to its image, so must $f_0^*$ map $E$ isomorphically to its image. 
\end{proof}
\par
We will see how we can apply this idea in \S \ref{S:sec3.a} for detecting nonvanishing subalgebras or subgroups of characteristic cohomology, with examples in \S \ref{S:sec7} and with more complete applications for matrix singularities in Part II of this paper.
\begin{Remark}
\label{Rem1.13}
Although $\cB_{\cV}(f_0, R)$ is not functorial, it does satisfy a relation involving a type of relative Gysin homomorphism, where in place of Poincare duality, Alexander duality is used because the links are not manifolds.  For a morphism $\varphi : \cW_0, 0 \to \cV_0, 0$ defined by $\tilde{\varphi}$ we have a map for sufficiently small $0 < \eta << \gevar_2, \gevar_1$ so that $\tilde{\varphi}(\overline{B}_{\gevar_2}) \subset B_{\gevar_1}$ and 
$f_0(\overline{B}_{\gevar_1}) \subset B_{\gd}$.  Then, 
\begin{align}
\label{Eqn1.14}
 \widetilde{H}^j(S^{2s-1}_ {\gevar_2} \cap \cW_0 ; \bk) &\overset{\ga}{\simeq} \widetilde{H}_{2s-2 - j}(S^{2s-1}_ {\gevar_2} \backslash \cW_0; \bk) \overset{j_{\gevar_2\, *}}{\longrightarrow}  \widetilde{H}_{2s-2 - j}(B_ {\gevar_2} \backslash \cW_0; \bk) \overset{\tilde{\varphi}_*}{\longrightarrow}  \notag  \\ 
\widetilde{H}_{2s-2 - j}(B_ {\gevar_1} \backslash \cV_0; \bk) &\simeq \widetilde{H}_{2s-2 - j}(S^{2n-1}_ {\gevar_1} \backslash \cV_0; \bk) \overset{\ga^{-1}}{\simeq}  \widetilde{H}^{2(n-s) + j}(S^{2n-2}_ {\gevar_1} \cap \cV_0 ; \bk)  
\end{align}


The composition in \eqref{Eqn1.14} yields
$\widetilde{H}^j(S^{2s-1}_ {\gevar_2} \cap \cW_0 ; \bk) \to \widetilde{H}^{2(n-s) + j}(S^{2n-2}_ {\gevar_1} \cap \cV_0 ; \bk)$.  
Then, via the identification for different $\gevar_i$, we obtain a form of {\em relative Gysin homomorphism} 
\begin{equation}
\label{Eqn1.15}
\varphi_* :  \widetilde{H}^j(L(\cW_0) ; \bk) \longrightarrow \widetilde{H}^{2(n-s) + j}(L(\cV_0) ; \bk)\, .  
\end{equation}
Also, by choosing consistent bases for the cohomology, this will induce a 
Gysin-type homomorphism
$\cB_{\cW}(g_0 ; \bk) \to \cB_{\cV}(f_0 ; \bk)$, which shifts degrees by $2(n-s)$.
  
\end{Remark}
\section{$\cK_H$ and $\cK_{\cV}$ Invariance of Characteristic Cohomology} 
\label{S:sec2a}
\par
We next turn to the invariance properties of the characteristic cohomology.
\subsection*{Invariance of Characteristic Cohomolgy $\cA_{\cV}(f_0; R)$ under 
$\cK_H$ Equivalence} \hfill 
\par 
\par
The dependence of $\cA_{\cV}(f_0; R)$ on $f_0$ is clarified by the next proposition.  
\begin{Proposition}
\label{Prop2.4a}
Suppose $f_i : \C^n, 0  \to \C^N, 0$, $i = 1, 2$ are $\cK_H$--equivalent. Let $F_i$, $i = 1, 2$, denote the Milnor fibers of $H \circ f_i$ for a $w \in B_{\eta}^*$.  Then, for any coefficient ring $R$, there is a cohomology algebra isomorphism $\ga : H^*(F_1; R) \simeq H^*(F_2; R)$ such that 
$\ga(\cA_{\cV}(f_1; R)) = \cA_{\cV}(f_2; R)$.  \par 
Hence, the structure of the cohomology of the Milnor fiber of $H \circ f_0$ as a 
graded algebra (or graded module) over $\cA_{\cV}(f_1; R)$ is, up to isomorphism, independent of the $\cK_H$--equivalence class of $f_0$.
\end{Proposition}
\par.
\begin{proof}[Proof of Proposition \ref{Prop2.4a}]
\par 
By the $\cK_H$--equivalence of the germ $f_i : \C^n, 0  \to \C^N, 0$, there are representatives $f_i : U  \to W$, for open neighborhoods $U$ and $W$, and a diffeomorphism onto a subspace
\begin{align}
\label{Eqn2.4a}
\Phi : U^{\prime} \times W^{\prime} &\to U \times W  \\
(x, y)\qquad  &\mapsto \qquad (\varphi(x), \varphi_1(x, y)) \notag
\end{align}
sending $(0, 0) \mapsto (0, 0)$ such that $\Phi$ preserves $H\circ pr_2$ for $pr_2 : \C^n \times \C^N$ the projection onto the second factor, and so that 
$f_2(\varphi(x)) = \varphi_1(x, f_1(x))$ for all $x \in U^{\prime}$.  Thus, $\Phi(\graph(f_1)) = \graph(f_2) \cap Im(\Phi)$, and for any Milnor fiber $F_w$ of $H$, 
$\Phi(\C^n \times F_w) = (\C^n \times F_w) \cap Im(\Phi)$.  \par
We let $g_i = H\circ f_i$.  Next we choose $0 < \eta_1 << \gd_1 << \gevar_1$ so that \begin{itemize}
\item[i)] $B_{\gd_1} \subset W^{\prime} $;
\item[ii)] $B_{\gevar_1} \subset U^{\prime}$;
\item[iii)]  the Milnor fibration of $H$ is given by $H : H^{-1}(B_{\eta_1}^*) \cap B_{\gd_1} \to B_{\eta_1}^*$; and 
\item[iv)] the Milnor fibration of each $g_i$ is given by $g_i : g_i^{-1}(B_{\eta_1}^*) \cap B_{\gevar_1} \to B_{\eta_1}^*$.
\end{itemize}
\par 
Next, we begin to find a series of $(\eta_j, \gd_j, \gevar_j)$ so that:
\begin{itemize}
\item[1)] $0 < \eta_{j+1} < \eta_j$; $0 < \gd_{j+1} < \gd_j$, and $0 < \gevar_{j+1} < \gevar_j$ and 
\begin{equation}
\label{Eqn2.5a}
 g_i : g_i^{-1}(B_{\eta_{j+1}}^*) \cap B_{\gevar_{j+1}} \to B_{\eta_{j+1}}^* \quad \text{ and } \quad H : H^{-1}(B_{\eta_{j+1}}^*) \cap B_{\gd_{j+1}} \to B_{\eta_{j+1}}^* 
\end{equation} 
are Milnor fibrations for $g_i$, $i = 1, 2$, resp. $H$.  \par
\item[2)]  
\begin{equation}
\label{Eqn2.5b}
 T_{j+1} \,\, \overset{def}{=} \,\, \Phi(B_{\gevar_{j+1}}  \times B_{\gd_{j+1}}) \,\, \subset \,\, B_{\gevar_j} \times  B_{\gd_j} \, .
\end{equation}  
\item[3)] 
$(B_{\gevar_{j+1}}  \times B_{\gd_{j+1}}) \subset T_j$ (as $T_j$ is an open neighborhood of $(0, 0)$).   
\item[4)] If $\cV^{(i, j)}_w$ denotes the Milnor fiber of $g_i : g_i^{-1}(B_{\eta_j}^*) \cap B_{\gevar_j} \to B_{\eta_j}^*$, then for $w \in B_{\eta_{j+1}}^*$  the inclusions of Milnor fibers in \eqref{Eqn2.7b} are homotopy equivalences.
\begin{equation}
\label{Eqn2.7b}
\cV^{(i, j+1)}_w \,\, \subset \,\, \cV^{(i, j)}_w \, ;
\end{equation}  
\item[5)] We repeat these steps for $j = 1, \dots , 4$. 
\end{itemize}
\par
We observe that as both $\Phi$ and the graph maps are diffeomorphisms,  
$\Phi : \cV^{(i, j)}_w  \simeq  \Phi(\graph(f_i)(\cV^{(i, j)}_w))$.  We choose a 
$w \in B_{\eta_4}^*$ and let $Y_j  = \graph(f_2)((\cV^{(2, j)}_w)$ and $Z_j = \Phi(\graph(f_1)((\cV^{(1, j)}_w))$.  Consider the sequence of inclusions and mapping 
\begin{equation}
\label{Eqn2.8b} 
Z_4 \subset Y_3 \,\, \subset Z_2 \,\, \subset \,\, Y_1 \,\, \overset{H}{\rightarrow} \,\, (H^{-1}(w) \cap B_{\eta_4}) 
\end{equation}
Then, for cohomology (with coefficients in $R$ understood)
\begin{equation}
\label{Eqn2.9a} 
 H^*(H^{-1}(w) \cap B_{\eta_4}) \overset{H^*}{\rightarrow} H^*(Y_1) \rightarrow H^*(Z_2) \rightarrow H^*(Y_3) \rightarrow H^*(Z_4)
\end{equation}
Now the composition $H^*(Y_1) \rightarrow H^*(Z_2) \rightarrow H^*(Y_3)$ is an isomorphism; hence $H^*(Z_2) \rightarrow H^*(Y_3)$ is surjective.  Second, the composition $H^*(Z_2) \rightarrow H^*(Y_3) \rightarrow H^*(Z_4)$ is also an isomorphism so $H^*(Z_2) \rightarrow H^*(Y_3)$ is one-one.  Thus, $H^*(Z_2) \rightarrow H^*(Y_3)$ is an isomorphism.  Hence, so are the other inclusions isomorphisms.  \par
A similar argument for the Milnor fibers of $H$ for the various $j$, together with $\Phi$ preserving $H^{-1}(w)$ implies that $\Phi^*$ induces an isomorphism of the cohomology of the Milnor fiber.  Since the map $\Phi : \graph(f_1)(\cV^{(1, 2)}_w) \to \,\, \graph(f_2)(\cV^{(2, 1)}_w)$ commutes with $H$, we deduce that the induced isomorphism from $\Phi^*$ preserves the subalgebra $pr_2^*(H^*(H^{-1}(w)) \cap B_{\gd})$.  By the isomorphism on cohomology via $\graph(f_i)^*$, we obtain the preservation of the characteristic subalgebra.  
\end{proof}
\begin{Remark}
\label{Rem2.9b}
We can apply the preceding argument for the sequence of inclusions in \eqref{Eqn2.8b} to conclude $\Phi : \graph(f_1)(\cV^{(1, 2)}_w) \to \,\, \graph(f_2)(\cV^{(2, 1)}_w)$ induces an isomorphism for both integer homology and the fundamental group.  As the closures of both of these spaces are smooth manifolds with boundaries and hence have CW-complex structures, it follows by the Hurewicz and Whitehead theorems that the restriction of $\Phi$ is a homotopy equivalence.  
\end{Remark}
\subsection*{Invariance of Characteristic Cohomology $\cC_{\cV}(f_0; R)$ and 
$\cC_{\cB}(f_0; R)$ under $\cK_{\cV}$ Equivalence} \hfill 
\par 
In analogy with Proposition \ref{Prop2.4a}, the dependence of $\cC_{\cV}(f_0; R)$ and 
$\cB_{\cV}(f_0; R)$ on the $\cK_{\cV}$-equivalence class of $f_0$ is given by the next proposition.  
\begin{Proposition}
\label{Prop2.10a}
Suppose $f_i : \C^n, 0  \to \C^N, 0$, $i = 1, 2$ are $\cK_{\cV}$--equivalent.  Let $\cV_i = f_i^{-1}(\cV)$.  Then, for any coefficient ring $R$, there is a cohomology algebra isomorphism $\gb : H^*(\C^n \backslash \cV_1; R) \simeq H^*(\C^n \backslash \cV_2; R)$ such that $\gb(\cC_{\cV}(f_1; R)) = \cC_{\cV}(f_2; R)$.  \par 
Hence, the structure of the cohomology of the complement $\C^n \backslash \cV_i$ as a graded algebra (or graded module) over $\cC_{\cV}(f_1; R)$ is, up to isomorphism, independent of the $\cK_{\cV}$--equivalence class of $f_i$.
\end{Proposition}
\begin{proof}
The proof is similar to that for Proposition \ref{Prop2.4a}, except that the diffeomorphism 
$\Phi : U^{\prime} \times W^{\prime} \to U \times W $ in \eqref{Eqn2.4a} only preserves 
$\C^n \times \cV$.
\end{proof}
\par
Then, for links we have a corresponding result provided the coefficient ring $R = \bk$, a field of characteristic $0$.
\begin{Proposition}
\label{Prop2.11a}
Suppose $f_i : \C^n, 0  \to \C^N, 0$, $i = 1, 2$ are $\cK_{\cV}$--equivalent.  Let $\cV_i = f_i^{-1}(\cV)$.  Then, there is an isomorphism of graded vector spaces $\gb : H^*(L(\cV_1); \bk) \simeq H^*(L(\cV_2); \bk)$ such that $\gb(\cB_{\cV}(f_1; \bk)) = \cB_{\cV}(f_2; \bk)$.  \par 
Hence, $\cB_{\cV}(f_i; \bk)$ is, up to isomorphism, independent of the 
$\cK_{\cV}$--equivalence class of $f_i$.
\end{Proposition}
\begin{proof}
The diffeomorphism $\Phi : U^{\prime} \times W^{\prime} \to U \times W$ induces diffeomorphisms $\graph(f_i) \cap \cV$ and $\graph(f_i) \backslash \cV$.  These induce diffeomorphisms  $U \cap \cV_1 \simeq U^{\prime} \cap \cV_2$ and $U \backslash \cV_1 \simeq U^{\prime} \backslash \cV_2$.  These first induce isomorphisms 
$H^*(U^{\prime} \backslash \cV_2; \bk) \simeq H^*(U \backslash \cV_1; \bk)$.    This continues to hold for sufficiently small balls using the argument in the proof of Proposition \ref{Prop2.4a}. As the homeomorphisms commute with $f_i^*$, we obtain the restriction isomorphism $\cC_{\cV}(f_1; \bk) \simeq \cC_{\cV}(f_2; \bk)$.  
\par 
Then, by choosing corresponding bases for these cohomology groups we obtain via the Kronecker pairings, isomorphisms with the homology groups of the complements. Then, associated to the isomorphisms between the $\cC_{\cV}(f_i; \bk)$, there is an induced isomorphism in reduced homology $\widetilde{\gG}_{\cV}(f_1; \bk) \simeq \widetilde{\gG}_{\cV}(f_2; \bk)$.  
Lastly, Alexander duality induces isomorphisms of graded vector spaces $\cB_{\cV}(f_1; \bk) \simeq \cB_{\cV}(f_2; \bk)$.  
\end{proof}

\section{Detecting the Nonvanishing of Characteristic Cohomology} 
\label{S:sec3.a}
\par 
We next ask for a singularity $\cV_0$ of type $\cV$, what will be the nonvanishing parts of the characteristic subalgebras $\cA_{\cV}(f_0; R)$, $\cC_{\cV}(f_0; R)$ and the characteristic cohomology $\cB_{\cV}(f_0; R)$?  For the Milnor fiber, $\cA_{\cV}(f_0; R)$  is isomorphic to a quotient algebra of $H^*(F_w ; R)$, but possibly it is just $H^0(\cV_w ; R)$.  Similarly, for the complement $\cC_{\cV}(f_0; R)$, it is isomorphic to a quotient of $H^*(\C^N \backslash \cV ; R)$; and then we can determine a nonzero subgroup in $\cB_{\cV}(f_0; R)$ via Alexander duality.  
\par 
We give a general method for detecting such non-zero subgroups of characteristic cohomology using \lq\lq vanishing compact models\rq\rq\, for both the Milnor fiber and complement.  \par
\subsection*{Nonvanishing Characteristic Cohomology for the Milnor Fiber} \hfill
\par
We consider a hypersurface singularity $\cV, 0 \subset \C^N, 0$ with defining equation $H : \C^N, 0 \to \C, 0$ and Milnor fibration $H : H^{-1}(B_{\eta}^*) \cap B_{\gd_0} \to B_{\eta}^* $.
\begin{Definition}
\label{Def2b.1}
We say that $\cV, 0$ has a {\em vanishing compact model for its Milnor fiber} if there is a compact space $Q_{\cV}$, smooth curves $\gg : [0, \eta) \to B_{\eta}$ satisfying $|\gg(t)| = t$ and $\gb : [0, \eta) \to [0, \gd_0)$, monotonic with $\gb(0) = 0$, and an embedding into the Milnor fibration of $H$, \par
$\Phi : Q_{\cV} \times (0, \gd) \hookrightarrow H^{-1}(B_{\eta}^*) \cap B_{\gd}$ such that: 
\begin{itemize}
\item[i)]  each $H : H^{-1}(\overline{B}_{|\gg(t)|}^*) \cap B_{\gb(t)} \to \overline{B}_{|\gg(t)|}^*$ is again a Milnor fibration for $H$
\item[ii)]  each $\Phi(Q_{\cV} \times \{ t\}) \subset F_w$ is a homotopy equivalence for $F_w$ the Milnor fiber of i) over $w = \gg(t)$. 
\end{itemize}
\end{Definition}
This is the analogue of a basis of smoothly vanishing cycles for the isolated hypersurface case.  \par 
Next, with the situation as above, let $E \subseteq H^*(Q_{\cV} ; R)$ be a graded subgroup.  We say that a compact subspace with inclusion map  $\gl_E : Q_{E} \subseteq Q_{\cV}$ detects $E$ in cohomology with $R$ coefficients if the map on cohomology 
$\gl_E^* : H^*(Q_{\cV}; R)  \to H^*( Q_{E}; R)$ induces an isomorphism from 
$E$ to $H^*( Q_{E}; R)$.  Then, we say that a {\em germ of an embedding $\iti_E : \C^s, 0 \to \C^N, 0$ detects $E$} if for sufficiently small $0 < \eta << \gevar < \gd$ there is a vanishing compact model $\Psi : Q_{E} \times (0, \gd) \hookrightarrow (H\circ \iti_E)^{-1}(B_{\eta}^*) \cap B_{\gevar}$ for the Milnor fibration of $H \circ \iti_E$ so that $\iti_E \circ \Psi = \Phi \circ (\gl_E \times id)$, i.e. \eqref{CD3.1} commutes.  
\begin{equation}
\label{CD3.1}
\begin{CD}
   {Q_{E} \times (0, \gd)} @> {\Psi} >> {(H\circ \iti_E)^{-1}(B_{\eta}^*) \cap B_{\gevar}}\\
   @V{\gl_E \times id}VV @V {\iti_E}VV \\
{Q_{\cV} \times (0, \gd)} @>{\Phi}>> {H^{-1}(B_{\eta}^*) \cap B_{\gd}}  \\
\end{CD} 
\end{equation}
\par
We then have the simple Lemma.
\begin{Lemma}[Detection Lemma for Milnor Fibers]
\label{Lem2.4a}
Given $f_0 : \C^n, 0 \to \C^N, 0$ defining $(\cV_0, 0)$ of type $\cV$, suppose there is a germ $g : \C^s, 0 \to \C^n, 0$ such that $f_0 \circ g$ is $\cK_H$-equivalent to a germ detecting $E$. Then, $\cA_{\cV}(f_0, R)$ contains a graded subgroup which is isomorphic to $E$ via $f_0^* : H^*(F_w; R) \to H^*(\cV_w; R)$ to $E$.  
\end{Lemma}
\begin{proof}
\par
We use the functoriality of $g^* : \cA_{\cV}(f_0, R) \to \cA_{\cV}(f_0\circ g, R)$ given by Lemma \ref{Lem1.10}.  
We do so using the representation by \eqref{Eqn1.11}, with $\tilde{\varphi}$ representing $g$; and we first consider the case where the composition $f_0 \circ \tilde{\varphi}$ denotes $f_0 \circ g = \iti_E$.  Provided $w = \gg(t)$, with $0 < |w| < \eta$ is sufficiently small, there is the compact model $\Phi (Q_{\cV} \times \{ t\}) \subset F_w$.  The composition gives as the embedding $\iti_E : (Q_E \times \{ t\})  \subset (Q_{\cV} \times \{ t\}) \subset F_w$.  In cohomology it maps  $E \subseteq H^*(F_w; R)$ isomorphically to $H^*(Q_E \times \{ t\}; R) \simeq H^*(Q_E; R)$.  \par
Then, if we compose the corresponding version of \eqref{Eqn1.11}, the map on cohomology factors through $(Q_E \times \{ t\}) \subset S_w$ (for $\cS_w$ the Milnor fiber of $H\circ \iti_E$).  It will then send $E \subseteq H^*(F_w; R)$ isomorphically to the subgroup of the intermediate cohomology $H^*(\cV_w; R)$.  Thus, $\cA_{\cV}(f_0, R)$ contains this isomorphic copy of $E$ via $\tilde{f}_{0, w}$.  \par
Second, if instead $f_0 \circ g$ is $\cK_H$-equivalent to $\iti_E$, by Proposition \ref{Prop2.4a}, there is an algebra isomorphism $\cA_{\cV}(\iti_E, R) \simeq \cA_{\cV}(f_0\circ g, R)$.  Then, $\cA_{\cV}(f_0\circ g, R)$ contains a subspace isomorphic under an algebra isomorphism to $E$. Since this subspace is, up to an algebra isomorphism, the image of $g^*$ of the image of $\tilde{f}_{0, w}^*(E)$ that image must be an isomorphic image of $E$.  
\end{proof}
\par
\subsection*{Nonvanishing Characteristic Cohomology for the Complement and Link} \hfill
\par
With the above notation, we consider the characteristic cohomology of the complement and link.  We use the notation and neighborhoods given in the definition of the characteristic cohomology for the complement and link in \S \ref{S:sec1} for $f_0(\overline{B}_{\gevar_0}) \subset B_{\gd_0}$.  Then, 
$$ f_0^* : H^*(\overline{B}_{\gd_0}\backslash \cV ;R) \longrightarrow H^*(\overline{B}_{\gevar} \backslash \cV_0 ; R) \, . $$
We introduce a corresponding vanishing compact model for the complement. 
\begin{Definition}
\label{Def2b.2}
We say that $\cV, 0$ has a {\em vanishing compact model for the complement} if there is a compact space $P_{\cV}$, a smooth curve $\gg : [0, \gd) \to [0, \gd_0)$, monotonic with $\gg(0) = 0$, and an embedding into the complement of $\cV$, \par
$\Phi : P_{\cV} \times (0, \gd) \hookrightarrow B_{\gd_0} \backslash \cV$ such that: 
\begin{itemize}
\item[i)]  each $(\overline{B}_{\gg(t)}, \overline{B}_{\gg(t)} \cap \cV)$ again has a cone structure;
\item[ii)]  $(\overline{B}_{\gg(t^{\prime})}, \overline{B}_{\gg(t^{\prime})} \cap \cV) \subset (\overline{B}_{\gg(t)}, \overline{B}_{\gg(t)} \cap \cV)$ is a homotopy equivalence for $0 < t^{\prime} < t$; and
\item[iii)]  each $\Phi(P_{\cV} \times \{ t\}) \subset B_{\gg(t)} \backslash \cV$ is a homotopy equivalence.  
\end{itemize}
\end{Definition}
\par 
Next, with the situation as above, let $E \subseteq H^*(P_{\cV}; R)$ be a graded subgroup.  We say that a compact subspace with inclusion map  $\gs_E : P_E \subseteq P_{\cV}$ detects $E$ in cohomology with $R$ coefficients if the map on cohomology 
$\gs_E^* : H^*(P_{\cV}; R)  \to H^*(P_{E}; R)$ induces an isomorphism from 
$E$ to $H^*( P_{E}; R)$.  Then, we say that a {\em germ of an embedding $\itj_E : \C^s, 0 \to \C^N, 0$ detects $E$} if for sufficiently small $0 < \gevar < \gd$ with 
$\iti_E(B_{\gevar}) \subset B_{\gd}$, there is a vanishing compact model $\Psi : P_{E} \times (0, \gd) \hookrightarrow B_{\gevar} \backslash \itj_E^{-1}(\cV)$ so that $\itj_E \circ \Psi = \Phi \circ (\gs_E \times id)$.  We then have the simple Lemma.
\begin{Lemma}[Second Detection Lemma]
\label{Lem2.5a}
Given $f_0 : \C^n, 0 \to \C^N, 0$ defining $\cV_0, 0$ of type $\cV$, suppose there is a germ $g : \C^s, 0 \to \C^n, 0$ such that $f_0 \circ g$ is $\cK_{\cV}$-equivalent to a germ detecting $E$.  Then, $\cC_{\cV}(f_0, R)$ contains a graded subgroup which is isomorphic to $E$ via 
$f_0^* : H^*(B_{\gd} \backslash \cV; R) \to H^*(B_{\gevar} \backslash \cV_0); R)$.  
\end{Lemma}
\begin{proof}
\par
The proof is similar to that for Lemma \ref{Lem2.4a} using instead the functoriality of $g^* : \cC_{\cV}(f_0, R) \to \cC_{\cV}(f_0\circ g, R)$ given by Lemma \ref{Lem1.10}.  As we are really working with local cohomology, we must consider the cohomology groups of complements on varying neighborhoods of $0$. 
By assumption there are vanishing compact models: 
$\Phi : P_{\cV} \times (0, \gd) \hookrightarrow B_{\gd_0} \backslash \cV$ and 
$\Psi : P_{E} \times (0, \gd) \hookrightarrow B_{\gevar} \backslash \iti_E^{-1}(\cV)$ so that 
$\itj_E \circ \Psi = \Phi \circ (\gs_E \times id)$.  
\par 
 Provided $0 < \gg(t) < \gd$ for $\gd$ sufficiently small, there is the compact model 
$\Phi (P_{\cV} \times \{ t\}) \subset B_{\gg(t)} \backslash \cV$.  The composition 
$\itj_E : (P_E \times \{ t\})  \subset (P_{\cV} \times \{ t\})$ is an embedding which in cohomology maps $E \subseteq H^*(P_{\cV} \times \{ t\}; R)$ isomorphically to $H^*(P_E \times \{ t\}; R) \simeq H^*(P_E; R)$.  \par
Then, we refer to corresponding version of \eqref{Eqn1.11} with $\tilde{\varphi}$ representing $g$ and the composition $f_0 \circ \tilde{\varphi}$ denoting 
$f_0 \circ g = \itj_E$.  We see that this composition further composed with the map on cohomology induced from $(P_E \times \{ t\}) \subset B_{\gevar} \backslash \iti_E^{-1}(\cV)$ will then send $E \subseteq H^*(B_{\gd_0} \backslash \cV; R)$ isomorphically to the graded subgroup of the intermediate cohomology 
$H^*(B_{\gevar} \backslash \cV_0; R)$.  Thus, $\cC_{\cV}(f_0, R)$ contains this isomorphic copy of $E$ via $f_0^*$.  \par
Also, if instead $f_0 \circ g$ is $\cK_H$-equivalent to $\iti_E$, by Proposition \ref{Prop2.4a}, there is an algebra isomorphism $\cC_{\cV}(\iti_E, R) \simeq \cC_{\cV}(f_0\circ g, R)$.  Then, $\cC_{\cV}(f_0\circ g, R)$ contains a subspace isomorphic under an algebra isomorphism to $E$. Since this subspace is, up to an algebra isomorphism, the image by $g^*$ of the image of $f_0^*(E)$, that image must be the isomorphic image of $E$.  
\end{proof}
\begin{Corollary}
\label{Cor2.5a}
Given $f_0 : \C^n, 0 \to \C^N, 0$ defining $\cV_0, 0$ of type $\cV$, suppose there is a germ $g : \C^s, 0 \to \C^n, 0$ such that $f_0 \circ g$ is $\cK_{\cV}$-equivalent to a germ detecting $E \subseteq \widetilde{H}^*(\C^N\backslash \cV; \bk)$, for $\bk$ a field of characteristic $0$.  Then, 
$\cB_{\cV}(f_0, \bk)$ contains a graded subgroup which is isomorphic via the Kronecker pairing and Alexander duality to the image of $E$ via the isomorphisms
$\widetilde{H}^j(B_{\gevar} \backslash \cV_0; \bk) \simeq \widetilde{H}^{2n - 2 -j} (S_{\gevar}^{2n-1} \cap \cV_0; \bk)$.  
\end{Corollary}
\begin{proof}
This is a consequence of the Second Detection Lemma \ref{Lem2.5a} and the definition of $\cB_{\cV}(f_0, \bk)$ via the Kronecker pairing and Alexander duality.  
\end{proof}

\section{Module Structure over Characteristic Cohomology for the Cohomology of Milnor Fibers} 
\label{S:sec7}
\par
In the hypersurface case we can consider the module structure of the cohomology of the Milnor fiber over the characteristic cohomology subalgebra.  
\par
We first consider two examples at the opposite extremes for matrix singularities 
Let $\cV_0 = f_0^{-1}(\cV)$ be defined by $f_0 : \C^n, 0 \to \C^N, 0$ for 
$\cV, 0 \subset \C^N, 0$. We illustrate how the characteristic subalgebra together with the topology of the \lq\lq singular Milnor fiber\rq\rq\, of $f_0$ contributes to the Milnor fiber cohomology, including the module structure, of $\cV_0$.
\begin{Example}
\label{Exam7.1}
  There are two cases at opposite extremes for singularities defined by $f_0 : \C^n, 0 \to \C^N, 0$ which is transverse off $0$ to $\cV$.  These are either $n < \codim(\sing(\cV))$ versus $f_0$ is the germ of a submersion.  In the first case, when 
$n < k = \codim(\sing(\cV))$, then $\cV_0$ has an isolated singularity, and the singular 
Milnor fiber for $f_0$ is diffeomorphic to the Milnor fiber for $\cV_0$, so the Milnor number of $\cV$ and the singular Milnor number of $f_0$ agree.  Also, by the result of Kato-Matsumoto \cite{KMs} the Milnor fiber of $\cV$ is $N - 2- (N- k)= k - 2$ connected.
Thus, $\cA^{(*)}(f_0, R) = H^0(\cV_w; R) \simeq R$.  
As the Minor fiber is homotopy equivalent 
to a CW-complex of real dimension $n-1$, the corresponding classes which occur for 
the Milnor fiber will have a trivial module structure over $\cA^{(*)}(f_0, R)$.  \par
Second, if $f_0$ is the germ of a submersion, then the Milnor fiber has the form 
$F_w \times \C^k$, where $F_w$ is the Milnor fiber of $\cV$ for $k = n - N$.  Thus, the Milnor fiber of $\cV_0$ has the same cohomology as $F_w$.   
We conclude that $f_0^* : H^*(\cV_w ; R)  \simeq H^*(F_w ; R)$, or 
$\cA^{(*)}(f_0, R) = H^*(F_w ; R)$.  Also, there are no singular vanishing cycles.  
\par 
Thus, for these two cases there is the following expression for the cohomology of the Milnor fiber, where the second summand has trivial module structure shifted by degree $n-1$.
\begin{equation}
\label{Eqn7.2}
   \quad H^*(\cV_{w}; R) \, \simeq \cA^{(*)}(f_0, R) \oplus R^{\mu}[n-1]   
\end{equation}
where $\mu = \mu_{\cV}(f_0)$ denotes the singular Milnor number for the corresponding hypersurface $\cV_0$.  
\end{Example}
\par 
\vspace{1ex} 
\flushpar
{\it Question: We ask how must \eqref{Eqn7.2} be modified for general singularities of given types?} 
\par
If $R$ is a field of characteristic $0$, then for a general hypersurface singularity we write \eqref{Eqn7.2} in a more general form as a direct sum.
\begin{equation}
\label{Eqn7.2a}
   \quad H^*(\cV_{w}; R) \, \simeq \cA_{\cV}(f_0, R) \oplus \cW_{\cV}(f_0, R)   
\end{equation}
We then may ask several questions about the properties of the summand 
$\cW_{\cV}(f_0, R)$.
\begin{itemize}
\item[i)] Does $R^{\mu}[n-1]$ for $\mu = \mu_{\cV}(f_0)$ occur as a summand?
\item[ii)] Does $\cW_{\cV}(f_0, R)$ vanish below degree $n-1$?
\item[iii)] If i) holds, is there an additional contribution in degree $n-1$ to 
$\cW_{\cV}(f_0, R)$?
\item[iv)] If ii) does not hold, can $\cW_{\cV}(f_0, R)$ be chosen to be an 
$\cA_{\cV}(f_0, R)$-submodule? 
\end{itemize}
One step in establishing i) is in the case that $\cV$ is a hypersurface defined by $H$ , which is $H$-holonomic; and $f_0$ has finite $\cK_H$-codimension.  
By the $H$-holonomic property, $f_0$ is transverse to $\cV$ in a neighborhood of $0 \in \C^n$.  Then, there is a stabilization of $f_0$,  $f_t : U \to M$ defined for $t \in (-\gg, \gg)$ for some $\gg > 0$, so that for 
$0 < | t | < \gg$ $f_t$ is transverse to $\cV$.  Since $H \circ f_t$ defines a hypersurface, it satisfies the Thom condition $a_f$ so for appropriate $0 < \eta << \gd$, we can stratify the mapping 
$$H \circ f_t | \overline{B}_{\gd} : (H \circ f_t)^{-1}(B_{\eta}) \cap \overline{B}_{\gevar}  \to B_{\eta} \, .$$  
Then, the system of tubes for the stratification provide a neighborhood $N_{\cV_t}$ of 
$\cV_t = f_t^{-1}(\cV) \cap B_{\gevar}$ and a retraction onto it (see e.g. \cite{M1}, \cite{M2}, or \cite{GDW}).  Given a Milnor fiber $\cV_w = (H\circ f_t)^{-1}(w) \cap \overline{B}_{\gevar}$, let $\pi$ denote the composition of the inclusion and the projection $\cV_w \subset N_{\cV_t} \to \cV_t$.  There is an induced homomorphism
\begin{equation}
\label{Eqn7.3}
\pi_* : H_*(\cV_w; R) \,\, \rightarrow \,\, H_*(f_t^{-1}(\cV) \cap \overline{B}_{\gevar}; R)\, .
\end{equation}
In the case $R$ is a field of characteristic zero as above, then if $\pi_*$ is surjective, the dual map in cohomology \eqref{Eqn7.4} is injective. 
\begin{equation}
\label{Eqn7.4}
\pi^* : H^*(f_t^{-1}(\cV) \cap \overline{B}_{\gevar}; R) \,\, \rightarrow \,\, H^*(\cV_w; R)\, . 
\end{equation}
  Thus, by a result of Damon-Mond \cite{DM}, which also holds in the 
$H$--holonomic case \cite{D1}, $f_t^{-1}(\cV) \cap \overline{B}_{\gevar}$ is homotopy equivalent to a bouquet of $\mu = \mu_{\cV}$ spheres of dimension $n-1$.  Thus, the injectivity of \eqref{Eqn7.4} gives the factor $R^{\mu}[n-1]$ in \eqref{Eqn7.2}.  \par 
This is just a first step in answering the above questions.  \vspace{1ex}
\flushpar 
{\it Partial Criterion for \eqref{Eqn7.2a}: } For the occurrence of $R^{\mu}[n-1] $ as a subspace of $\cW_{\cV}(f_0, R)$ in \eqref{Eqn7.2a} for a finitely $\cK_H$-determined germ it is sufficient that \eqref{Eqn7.3} is surjective.  \par
For the remaining questions, there are few special cases such as generic central hyperplane arrangements \cite{OR} and generic hypersurface arrangements \cite{Li} where the answer to ii) is positive.  However, there are significant additional contributions in degree $n-1$ to $\cW_{\cV}(f_0, R)$ (see \S \ref{S:sec11}).  
\par
\section{Detecting Characteristic Cohomology for various General Cases} 
\label{S:sec11}
\par 


\begin{table}
\begin{tabular}{|p{1.25in}|p{1.75in}|p{1.75in}|} \hline
\multicolumn{1}{|p{1.in}|}{ {\bf Singularity Type} } & 
\multicolumn{1}{c}{\bf \lq\lq Universal Singularity $\cV$ \rq\rq} & \multicolumn{1}{|p{1.5in}|}{\bf Singularities of type $\cV$} \\ \hline\hline
{\it Discrimants } & Discriminants of Stable Germs  & Discriminants of Finitely Determined Germs\\ \hline
{\it  Bifurcation Sets} & Bifurcation Sets of $\cG$-Versal Unfoldings & Bifurcation Sets of $\cG$-Finitely Determined Unfoldings \\ \hline
{\it Hyperplane Arrangements} & Special Central Hyperplane Arrangements & Generic Versions of Special Hyperplane Arrangements \\ \hline
{\it Hypersurface Arrangements} & Special Central Hyperplane Arrangements  & Hypersuface Arrangements of Special Type \\  \hline
{\bf Exceptional Orbit Hypersurfaces} & Defined by Linear Algebraic Group Representations with Open Orbits &  special types of determinantal arrangements:\\  \hline
{\it Quiver Discriminants} &  Discriminants for Quiver Representations of Finite Type & 
Discriminants from Mappings to Quiver Representation Spaces \\ \hline
{\it Cholesky-Type Factorizations} &  Discriminants for Cholesky-Type Factorizations & 
Discriminants for Cholesky-Type Factorizations for Matrix Families. \\ \hline
{\em Matrix Singularities} & Varieties of Singular $m \times m$ Matrices of Three Types and $m \times p$ Matrices for $m \neq p$ &  Matrix Singularities of any of these Types \\ \hline
\end{tabular}
\vspace*{0.2cm}
\caption{Examples of General Cases of Singularities of Given Types.}
\label{table:gen.cases}
\end{table}

\par
We provide special examples of the general case of singularities of type $\cV$, a hypersurface which represents a \lq\lq universal singularity type\rq\rq.  We summarize below the descriptions of several of the main classes of singularities of given universal singularity types in Table \ref{table:gen.cases}.  These were mentioned in the introduction and all of them have characteristic cohomology for Milnor fibers (in the hypersurface case), complements, and links.  We can ask to what extent the form of the characteristic cohomology has been identified for each of these cases and when can the nonvanishing part be determined?  We briefly comment on the cases and their relation with the results here.  
\par
\vspace{1ex}
\flushpar
\subsection*{Exceptional Orbit Hypersurfaces (yielding special determinantal arrangements):} \par
Given a complex representation $\rho : G \to GL_N(\C)$ of a connected linear algebraic group $G$ with open orbit $U$ in $\C^N$, the union of the orbits of positive codimension form the {\em exceptional orbit variety} $\cV \subset \C^N$.  Such a space was first investigated by Sato \cite{So}, also see \cite{SK}, and is called a prehomogeneous space (although he referred to $\cV$ as the \lq\lq singular set\rq\rq which would conflict with our general discussion).  If $\cV$ is a hypersurface we refer to it as the {\em exceptional orbit hypersurface}.  There are a number of important classes of singularities which arise in this manner.  \par
The varieties of singular $m \times m$ matrices of each type are exceptional orbit hypersurfaces for appropriate representations of $GL_m(\C)$.  Also the variety of singular $m \times p$ matrices with $m \neq p$ is an exceptional orbit variety for a representation of $GL_m(\C) \times GL_p(\C)$.  In part II of this paper we shall extensively study the matrix singularities for the corresponding varieties of singular matrices.  \par
Second, as listed in the tables, the form of the characteristic cohomology has been explicitly determined by the results in \cite{DP} and \cite{D3} for coefficients over a field of characteristics $0$.  This includes the cases of discriminants of quiver representation spaces of finite type and the discriminants for (modified) Cholesky-type factorizations.  Singularities of these types are given by special types of \lq\lq determinantal arrangements\rq\rq given in \cite{DP2}.  For these cases, compact models for Milnor fibers and complements are given as homogenenous spaces and can be used to define vanishing compact models in \cite{DP}.  Then, the tower structures given in \cite{DP2} can be used to give analogous versions of kite maps for these cases which can be used for detection criteria.  \par 
Third, in the pioneering work of Buchweitz and Mond, the representation spaces for quivers of finite type were considered in \cite{BM} and identified as prehomogeneous spaces for reductive groups for which the {\em quiver discriminants}, which are exceptional orbit hypersurfaces, provided a large class of linear free divisors.  There are also compact models as given in \cite{D3} which can be used to construct vanishing compact models.  Now the restrictions of the Dynkin diagrams and root structures in \cite{BM} need to be employed to define detection maps.  \par 
In both cases the details still have to be determined for identifying nonvanishing parts of the characteristic cohomology.  
\par
\vspace{1ex}
\flushpar
\subsection*{Central Hyperplane and Hypersurface Arrangements :} \hfill \par
For a central hyperplane arrangement $\cV \subset \C^N$, it follows from the work of Arnold \cite{A}, Brieskorn \cite{Br}, and Orlik-Solomon \cite{OS} (more generally \cite[Chaps. 3, 5]{OT}), there is an explicit description of the cohomology of the complement $H^*(\C^N\backslash \cV; \C)$ generated by $1$-forms corresponding to each hyperplane with combinatorially defined relations (in fact, by Brieskorn,  this holds for coefficients $\Z$ using the $\Z$-subalgebra on these generators).  
For a central hyperplane arrangement $\cV_0 \subset \C^n$ defined by a linear map $f_0 : \C^n, 0 \to \C^N, 0$ transverse to $\cV$ off $0 \in \C^n$, it then follows from transversality that the combinatorial conditions up to codimension $n-1$ continue to hold.  It follows that $H^*(\C^n\backslash \cV_0; \C) = \cC_{\cV}(f_0, \C)$. This then allows us to compute $\cB_{\cV}(f_0, \C)$ by adding relations in degree $n-1$ and above; and then $H^*(L(\cV_0); \C) = \cB_{\cV}(f_0, \C)$ can be explicitly computed.  \par
In the case that $f_0$ is nonlinear there is no general result for 
$\cC_{\cV}(f_0, \Z)$, although we know the form it has as the image 
$f_0^*(H^*(B^N_{\gd}\backslash \cV; \Z))$ for sufficiently small $\gd >0$.  
The problem for determining this image involves detecting the nonvanishing of the terms.  One result is obtained by Libgober \cite{Li} for the case where $\cV = \cB_N$, 
the Boolean arrangement.  The singularities $\cV_0$ are referred to by him as {\it isolated non-normal crossings} (INNC) (these are the same as hypersurface arrangements defined by a finitely $\cK_{\cB_N}$-determined germ $f_0$ \cite{D1}).  Then, 
$\C^N\backslash \cB_N$ is homotopy equivalent to a torus $T^N$ so 
$$ H^*(\C^N\backslash \cB_N; \Z) \,\, \simeq  \,\, \gL^*\Z< e_1, \dots , e_N>\, .$$
The result of Libgober \cite[Thm 2.2]{Li} gives results for the homotopy groups, which 
together with the relative Hurewicz Theorem and the universal coefficient theorem, 
implies that $\cC_{\cV}(f_0, \Z)$ contains $\gL^*\Z< e_1, \dots , e_N>$ up through degrees $\leq n-2$; while there is not an explicit formula for degree $n-1$, Libgober  does obtain results for this degree using properties of the \lq\lq characteristic variety of an INNC\rq\rq.  \par 
However, there does not exist a general result guaranteeing the nonvanishing of the characteristic cohomology for generic hypersurface arrangements based on a general central hyperplane arrangement.  In the case of complexified arrangements, the Salvetti complex, see e.g. \cite[\S 5.2]{OT}, provides a compact model for the complement, which then provides a vanishing compact model for the detection method.  Hence, detection maps can be defined by linear sections whose images contain appropriate subspaces of the Salvetti complex.  If $df_0(0) : \C^n \to \C^N$ contains 
a generic $k$-plane section, then $f_0$ plays the role of a detection map; and the detection method will imply that $H^*(\C^N\backslash \cB_N; \Z)$ will map 
isomorphically in degree $< k-1$ onto its image in the characteristic cohomology of 
the complement.  \par  
\subsubsection*{Milnor Fibers of Hyperplane and Hypersurface Arrangements:} 
\par
For the cohomology of the Milnor fiber of central hyperplane arrangements, there are basically very few results.  
For central generic arrangements, the cohomology has been determined by Orlik-Randell.  The Milnor fiber of $B_N$ has the homotopy type of a torus of dimension 
$N-1$ so its cohomology has the form $\gL^*\Z< e_1, \dots , e_{N-1}>$. 
Orlik Randell \cite[Thm 2.6]{OR} show that this maps isomorphically to 
$H^*(\cV_w; \Z)$ in degrees $< n-1$ and in degree $n-1$ the Betti number is 
$b_{n-1} = \binom{N-2}{n-2} + N \binom{N-2}{n-1}$.  It follows the characteristic cohomology contains all of $\gL^*\Z< e_1, \dots , e_{N-1}>$ up through degree $n-2$.  
\par
There is an analogous result for generic hypersurface arrangements, i.e. INNC, by a result of Libgober \cite[Prop. 4.6]{Li} which also implies that the characteristic cohomology of the Milnor fiber contains all of $\gL^*\Z< e_1, \dots , e_{N-1}>$ up through degree 
$n-2$.  However, he does not give an explicit formula for $b_{n-1}$.  
Both of these results use a covering representation of the Milnor fiber to carry out the computations.  This was extended by Cohen-Suciu \cite{CS} to more general hyperplane arrangements; however, their computation involves complexes of chains for local systems on the covering representation.  This allows them to compute explicitly the result for certain hyperplane arrangements in dimension $\leq 3$, but there are not general results. \par
These show that for the generic linear arrangements and hypersurface arrangements the characteristic cohomology for the Milnor fiber occupies all degrees below $n-1$, so for these cases the answer to question ii) (in \S \ref{S:sec7}) is positive.  
We also ask for the extent of the additional contribution to $\cW_{\cV}(f_0, R)$ in 
\eqref{Eqn7.2a}.  As $B_N$ is a linear free divisor, we can compute $\mu_{B_N}(f_0)$ 
using the calculations in \cite[\S 6]{D1}.  For the generic hyperplane arrangement case, we have $\mu_{B_N}(f_0) = \binom{N-1}{n}$.  Also, in degree $n-1$, the characteristic cohomology can contribute a subspace of dimension $\binom{N-1}{n-1}$.  Then, 
$b_{n-1}$ can be reexpressed in terms of these two dimensions by:
$b_{n-1} =  \binom{N-1}{n-1} + n \binom{N-1}{n}$.  It follows that if the characteristic cohomology contributes the full amount in degree $n-1$, then there is still an additional contribution to $\cW_{\cV}(f_0, R)$, beyond that from the singular Milnor fiber, of dimension $(n-1) \binom{N-1}{n}$.  This says that each singular vanishing cycle contributes $n$ vanishing cycles to the Milnor fiber.  This raises the question of how exactly this extra cohomology is realized geometrically.
\par
For a generic hypersurface arrangement $\cV_0, 0$ defined by a nonlinear map germ $f_0$ transverse to $B_N, 0$ off $0 \in \C^n$, there are less precise results, even though we know the form of $\cC_{\cV}(f_0, \C)$ and $\cB_{\cV}(f_0, \C)$ by the above.    
To detect nonvanishing contributions to the characteristic cohomology for the 
Milnor fiber using the method given here, requires vanishing compact models for the Milnor fiber which we do not have.  \par
We are able to give one type of example where we can explicitly see what occurs in cohomology degree $n-1$.
\begin{Example}
\label{Ex11.6}
We consider an isolated curve singularity $\cV_0, 0 \subset \C^2, 0$ defined by 
$f = f_1 \cdot f_2 \cdots f_k$ with each $f_j$ defining an isolated curve singularity 
$\cV_i$, so $\cV_0 = \cup_{i = 1}^{k} \cV_i$.  We can alternately consider $\cV_0$ as a generic hypersurface arrangement defined by $F = (f_1, \dots , f_k) : \C^2, 0 \to \C^k, 0$ for the Boolean arrangement $\cB_k \subset \C^k$.  We note that $\C^2$ lies below the dimension to which the result of Libgober applies.  \par
We can stabilize $F$ to $F_t = (f_{1\, t}, \dots , f_ {k\, t}) : U \to \C^k$ so in particular each $\cV_{j\, t} = f_{j\, t}^{-1}(0) \cap B_{\gevar}$ is a Milnor fiber for $f_j$ and the $\cV_{j\, t}$ pairwise intersect transversely.  Then, $\cV_{0\, t} = \cup_{i = 1}^{k} \cV_{j\, t}$ is the singular Milnor fiber for $F$.  It is homotopy equivalent to a bouquet of $\mu_{\cB_k}(F)$ $S^1$\rq s.  If $I(f_i, f_j)$ denotes the intersection number of $\cV_{i\, t}$ and $\cV_{j\, t}$.  A smooth nearby fiber of $f$ close to $\cV_{0\, t}$ adds one vanishing cycle for each intersection point.  Thus, the Milnor number of $f$ is given by
\begin{equation}
\label{Equ11.9}
  \mu(f) \,\, = \,\,  \mu_{\cB_k}(F) \, + \, \sum_{i < j} I(f_i, f_j) \, . 
\end{equation} 
Then, $\cB_k$ has a Milnor fiber which is homotopy equivalent to a $k-1$ torus and has the torus as a compact model.  Thus, the possible contribution to $\cA_{\cB_k}(F, \Z)$ in dimension $1$ would have rank $k-1$.  However, for most examples the sum of intersection numbers considerably exceeds $k-1$; thus, $\cW_{\cB_k}(F, \Z)$ must be considerably larger than the contribution from characteristic cohomology.  For example if $f = f_1\cdot f_2\cdot  f_3$ with the $f_i$ distinct generic quadrics, then $\mu(f) = 25$, by \cite[\S 6]{D1} $\mu_{\cB_3}(F) = 13$, and the sum of intersection numbers is $12$; while $3-1 = 2$.  Thus, most of the cohomology in dimension $1$ that does not come from the singular Milnor fiber lies outside of the characteristic cohomology.  \par
A basic question then is to determine geometrically what part of the characteristic cohomology exists in $\cW_{\cB_k}(F, \Z)$ and what geometrically accounts for the remainder.  Anatoly Libgober indicates that results from \cite{CNL} contribute to answering this question.  
\end{Example}
\subsection*{Discriminants and Bifurcation Sets :} \hfill
 \par 
There are only very limited results for the topological structure of the complement for either discriminants or bifurcation sets.  For the stable germs obtained by unfolding simple hypersurface singularities, the complement is a $K(\pi, 1)$ by results of Arnold and Brieskorn.  However, this does not continue to be always true for ICIS by H. Kn\"{o}rrer.  Also, there is an explicit basis for the cohomology of complements of discriminants of stable $A_k$-singularities by results of Fuks \cite{Fk}, Vainstein \cite{V} and for those of types $C$, and $D$ for functions on manifolds with boundaries, by Goryunov \cite{G}, \cite{G2}.  Hence, only for complements of discriminants of finitely determined germs of these types do we have the form for $\cC_{\cV}(f_0, \C)$.  Otherwise little is known about the characteristic cohomology for these singularities.  \par 
Also, there are many different equivalence groups $\cG$ in the holomorphic category , allowing additional features to be preserved such as (flags of) distinguished parameters, equivariant germs, diagrams of mappings, distinguished varieties, and restrictions to (flags of) subvarieties, etc.  These are geometric subgroups of $\cA$ or $\cK$.  Then, unfoldings of finitely $\cG$-determined germs are modeled as singularities of type the 
bifurcation sets of $\cG$-versal unfoldings.  These need not always be hypersurfaces; however, in many important cases they are.  For virtually all of these, the cohomology of the Milnor fiber (in the hypersurface case) and that of the complement and link is unknown.  Hence, even the form of the characteristic cohomology is unknown.  Because of such a great variety of possibilities, essentially nothing is known about the topology of bifurcation sets of unfoldings for any of these groups $\cG$.  
\par 
By contrast, many of the universal singularities have been shown to be 
(H-holonomic) free divisors, see e.g. the list in \cite{D1} and the additional work in e.g. \cite{GM} and \cite{DP2}.  Thus, for these we can compute the singular Milnor number to determine a possible contribution for the Milnor fiber using the results of the previous section.  \par 
Hence, all of the list of questions given for matrix singularities still remain to be resolved in these cases. 
\par 

\end{document}